\documentclass{amsart}
\usepackage{amsmath,amssymb}
\usepackage{graphicx,color}

\numberwithin{equation}{section}

\DeclareMathOperator{\sgn}{sgn}

\DeclareMathOperator{\Real}{Re}
\DeclareMathOperator{\Imag}{Im}
\DeclareMathOperator{\Arg}{Arg}
\newcommand\loc{_{\mathrm{loc}}}

\newcommand{\sobHh}{\dot{{{H}}}^1_{\al}}
\newcommand\HH{\sobHh}
\newcommand\HHo{{\dot{H}^1_{\al,0}}}

\newcommand{\La}{{L}^2_{\al}}

\numberwithin{equation}{section}

\newtheorem{thm}{Theorem}[section]
\newtheorem{proposition}[thm]{Proposition}
\newtheorem{corollary}[thm]{Corollary}
\newtheorem{lemma}[thm]{Lemma}

\theoremstyle{remark}
\newtheorem{remark}[thm]{Remark}
\theoremstyle{definition}

\def\Hat{\widehat{\phantom{a}}}
\def\th{^{\text{th}}}

\newcommand{\al}{\alpha}
\newcommand{\be}{\beta}

\newcommand{\ep}{\epsilon}

\newcommand{\si}{\sigma}
\newcommand{\te}{\theta}
\newcommand{\vp}{\varphi}

\newcommand{\De}{\Delta}
\newcommand{\Ga}{\Gamma}

\def\RR{\mathbb{R}}
\def\NN{\mathbb{N}}

\renewcommand\leq\leqslant
\renewcommand\geq\geqslant
\newcommand\pd\partial

\begin{document}
\title[Fractional wave operators via Dirichlet-to-Neumann
maps]{Fractional powers of the wave operator\\ via Dirichlet-to-Neumann
  maps\\ in anti-de Sitter spaces}
\author{Alberto Enciso, Mar\'ia del Mar Gonz\'alez and Bruno Vergara}
\maketitle

\begin{abstract}
We show that the fractional wave operator, which is usually studied in the
context of hypersingular integrals but had not yet appeared in
mathematical physics, can be constructed as the Dirichlet-to-Neumann
map associated with the Klein--Gordon equation in anti-de Sitter
spacetimes. Several generalizations of this relation will be discussed too.
\end{abstract}

\section{Introduction}

It is a classical result in potential theory that the
Dirichlet-to-Neumann map of the harmonic extension problem in the
upper half-space is given by the square root of the Laplacian. This
relation can be generalized to encompass all fractional powers of the
Laplacian, and has recently made a major impact in the theory of
nonlocal elliptic equations~\cite{Caffarelli}.

Indeed, for any $\al\in(0,1)$ this relation connects the multiplier
\begin{equation}\label{Dega1}
[(-\Delta)^\al f]\Hat(\xi):=|\xi|^{2\al}\widehat f(\xi)
\end{equation}
on $\RR^n$, which is a nonlocal operator of the form
\begin{equation}\label{Dega2}
(-\Delta)^\al f(x)=c_{n,\al}\int_{\RR^n}\frac{f(x)-f(x')}{|x-x'|^{n+2\al}}\, dx'\,,
\end{equation}
with a local elliptic equation in $n+1$ variables. Specifically, given a function $f$ on $\RR^n$ let us consider the function $u$ on $\RR^n\times\RR_+$ that solves the boundary value problem
\begin{gather}
\De_x u + \partial_{yy} u+\frac{1-2\al}y \partial_y u=0\quad \text{in }\RR^n\times\RR_+\,,\label{elliptic}\\
u(x,0)=f(x)\,,\qquad \lim_{y\to\infty}f(x,y)=0\,;\notag
\end{gather}
it was shown in~\cite{Caffarelli} that
\[
(-\Delta)^\al f(x)=c_{\al}\lim_{y\to0^+}y^{1-2\al}\partial_y u(x,y)\,.
\]
This reduces to the ordinary derivative $\partial_y u(x,0)$ in the case of the square root of the Laplacian ($\al=\frac12$), and in fact for all values of~$\al$ it provides the natural Neumann datum associated with the elliptic operator~\eqref{elliptic}.
Note that a generalized formula for higher powers $\alpha\in(0,\frac{n}{2})$ has been established in \cite{Mar,Chang-Yang}.

Our first objective in this paper is to derive an analogous relation
for fractional wave operators, by which we will always mean a fractional power of
the usual wave operator $\square:=\pd_{tt}-\De$ and not an evolution
equation driven by a fractional power of the
Laplacian or, more generally, a generator of a suitable semigroup (for
fine information on the latter in various contexts, cf.~\cite{Kemppainen-Sjogren-Torrea,Stinga-Torrea,Stinga-Torrea:heat}). Note that this cannot be seen as an analytic continuation of the elliptic case and that in fact several nontrivial choices need to be made, starting with the very definition of the fractional wave operator. This is due to the fact that the symbol of the wave operator, $|\xi|^2-\tau^2$, is not positive definite, so one cannot immediately define $\Box^\al$ through this quantity to the power of~$\al$ as in~\eqref{Dega1}, and can also be seen in the integral formula~\eqref{Dega2}, since formally replacing the squared Euclidean distance $|x-x'|^2$ by its Minkowskian counterpart
\[
|x-x'|^2-(t-t')^2
\]
in the denominator leads to an integral that is too singular to be
well defined. The analytic difficulties associated to dealing with
fractional powers of hyperbolic operators will become apparent already
in Section~\ref{section:powers}.

There is a considerable body of classical work on
fractional wave operators, which have never appeared naturally in a
physical problem but which are studied in detail in the theory of hypersingular integrals (see e.g.~\cite{Samko} and references therein). In particular, the fractional wave operator $\Box^\al$ is defined for all noninteger~$\al$ as the multiplier
\begin{equation}\label{Boxal}
\widehat{\Box^\al f} (\tau,\xi):=\sigma_\al(\tau,\xi)\,\widehat{f}(\tau,\xi)\,,
\end{equation}
where the symbol $\sigma_\al$ is defined as
\begin{equation}
	\sigma_\al(\tau,\xi):=(|\xi|^2-\tau^2)^{\al} \chi_+(\tau,\xi) + e^{i\pi\al\sgn(\tau)} (\tau^2-|\xi|^2)^{\al}\chi_-	(\tau,\xi),\label{symbol1}
\end{equation}
with $\chi_\pm(\tau,\xi)$ denotes the indicator function of the set $\pm (|\xi|^2-\tau^2)>0$. Equivalently, choosing the principal branch of the complex logarithm one can write
\begin{equation}
	\sigma_\al(\tau,\xi):=\lim_{\epsilon\to 0^+}\big(|\xi|^2-(\tau-i\epsilon)^2\big)^{\al}.\label{symbol2}
\end{equation}
It should be noted that this is in fact a natural definition of
$\Box^\al$ in the sense that $\si_\al(\tau,\xi)$ raised to the
power of $1/\al$ gives $|\xi|^2-\tau^2$, i.e., the symbol of the wave
operator.\\

Our second objective in this paper is of a more geometric and physical
nature. Specifically, we will show that the fractional wave operator,
as considered in the theory of hypersingular integrals, does arise in
gravitational physics as the
Dirichlet-to-Neumann map of the Klein--Gordon equation \eqref{KG} in
anti-de Sitter spaces. Hyperbolic equations for
fields in anti-de Sitter backgrounds with nontrivial data on their
conformal boundaries have attracted much attention over the last two
decades, especially in connection with the celebrated AdS/CFT
correspondence in string theory~\cite{Maldacena,Witten}. Indeed, this
conjectural relation establishes a connection between conformal field
theories in $n$ dimensions and gravity fields on an $(n +
1)$-dimensional spacetime of anti-de Sitter type, to the effect that
correlation functions in conformal field theory are given by the
asymptotic behavior at infinity of the supergravity
action. Mathematically, this involves describing the solution to the
gravitational field equations in $(n + 1)$~dimensions (which, in the
simplest case of a scalar field reduces to the Klein--Gordon equation) in
terms of a conformal field, which plays the role of the boundary data
imposed on the (timelike) conformal boundary.

To present the specifics of this connection in the simplest case,  let
us start by recalling the basic structure of the anti-de Sitter (AdS)
of dimension $n+1$, which is a Lorentzian manifold of negative
constant sectional curvature (which we set here to $-1$) and, as such, satisfies the Einstein
equations with a negative cosmological constant. For concreteness, we will consider the
AdS half-space (see e.g.~\cite{Bachelot} for a mathematical
analysis of the problem and \cite{Ballon}~for the physics of this space). The metric can be
written using Poincar\'e coordinates
$(t,x,y)\in\mathbb{R} \times \mathbb{R}^{n-1} \times \mathbb{R}_+$
as
\begin{equation}
g^+:=\dfrac{dt^2-dy^2-|dx|^2}{y^2},\label{metric1}
\end{equation}
where $|dx|^2$ denotes the standard flat metric
on~$\RR^{n-1}$.

Analytically, here one can take boundary conditions on
the set $y=0$ (which is conformal to the $n$-dimensional Minkowski
space) and decay conditions at $y=\infty$. This makes it the obvious
generalization of the half-space model of the $(n+1)$-dimensional
hyperbolic space, which one describes in terms of the Poincar\'e coordinates
$(y,x)\in\RR_+\times\RR^n$ through the metric
\[
	g^+_{\mathbb H}:=\frac{dy^2+|dx|^2}{y^2},
\]
which is the natural setting for \eqref{elliptic} after a conformal change.

While we prefer to stick to the AdS half-space in this
Introduction, let us mention that in Section~\ref{section:geometry} we
will also consider in detail the problem for the usual AdS space,
whose conformal timelike infinity is the cylinder
$\RR\times\mathbb S^{n-1}$.  This is sometimes referred to as the
global AdS space in the context of the AdS/CFT conjecture and can be
covered by two half-space AdS charts. The results are qualitatively
the same but the algebra is less transparent. We will also
encounter the same behavior when we analyze more general stationary
asymptotically anti-de Sitter metrics, again in
Section~\ref{section:geometry}.


Let us then consider the
Klein--Gordon equation with
parameter~$\mu$ in the AdS space~\eqref{metric1},
\begin{equation}\label{KG}
\Box_{g^+}\phi+\mu\phi=0\,,
\end{equation}
where $\Box_{g^+}$ is the wave operator associated with the AdS metric:
\[
	\Box_{g^+}\phi:=y^2(\partial_{tt}\phi-\De_x\phi-\partial_{yy} \phi)-(1-n)y \partial_y\phi.
\]
Physically, $\mu$ is the mass of the
particle modeled by the Klein--Gordon equation plus a negative contribution from the scalar
curvature of the underlying space~\cite[Section 4.3]{Wald}. For our
purposes, it is convenient to assume that the parameter
\[
\al:=\bigg(\frac{n^2}4+\mu\bigg)^{1/2}\,
\]
takes values in the interval $(0,\tfrac{n}{2})$.

We then have that Equation \eqref{KG}
can be rewritten as a wave equation
with coefficients singular at $y=0$,
\begin{align}\label{eqphi2}
\partial_{tt}\phi-\De_x\phi-\partial_{yy} \phi- \frac{1-n}{y}\partial_y\phi +\frac{4\al^2-n^2}{4y^2}\phi=0\,.
\end{align}
A simple look at the singularities of the equation
reveals that the solutions are expected to scale at conformal
infinity as $y^{\tfrac{n}{2}\pm \al }$.
Then, since we are assuming
$\al>0$, the natural Dirichlet condition for this problem is
\begin{equation}\label{Dirichlet}
\lim_{y\to 0^+} y^{\al-\frac n2}\phi(t,x,y)=f(t,x)\,.
\end{equation}
If one prefers to prescribe a Neumann
condition~\cite{Warnick}, one must instead impose
\[
\lim_{y\to 0^+}y^{1-2\al}\,\pd_y\big(y^{\al-\frac n2}\phi(t,x,y)\big)=h(t,x)\,,
\]
for $\alpha\in(0,1)$, and a generalized Neumann condition $\alpha\in(1,\frac{n}{2})$, as we will see below.
With this notation in place, the main result of the paper is that the
fractional wave operator $\Box^\al$ (in flat space) is the
Dirichlet-to-Neumann map associated to switching on a nontrivial
boundary datum in an anti-de Sitter space:

\begin{thm}\label{T.main}
For any function $f\in C^\infty_0(\RR^n)$, let $\phi$ be the solution of the Klein--Gordon equation \eqref{KG}
with this Dirichlet boundary condition (Eq.~\eqref{Dirichlet}) and trivial
initial data at $-\infty$: $\phi(-\infty,x,y)=\phi_t(-\infty,x,y)=0$. Assume moreover that the
mass parameter $\al$ takes values in $(0,1)$. Then for any smooth
enough function $f$ one has the identity
\[
\Box^\al f(t,x)=c_{\al}\lim_{y\to 0^+}y^{1-2\al}\,\pd_y\big(y^{\al-\frac n2}\phi(t,x,y)\big)
\]
for an explicit constant
$c_{\al}=-2^{2\al-1}\Ga(\al)/\Ga(1-\al)$. More generally, if $\alpha\in
(0,\frac n2)$ is not an integer and we write $\al=m+\al_0$, with
$m:=\lfloor \al\rfloor$ the integer part, then
 \begin{equation*}
	\square^\al f(t,x)=c_\al\lim\limits_{y\to 0^+}
        y^{2(1-\al_0)}\Big(\frac{1}{y}\, \pd_y\Big)^{m+1} \big(y^{\al-\frac n2}\phi(t,x,y)\big),\label{DNgen}
\end{equation*}
with $c_\al:=(-1)^{m+1} 2^{\al+\al_0-1}\Ga(\al)/\Ga(1-\al_0)$.
\end{thm}

Several remarks are in order. First, an elementary
observation is that this identity, which obviously applies to much
for general data by the bounds that we establish in this paper, implies that the well-known estimates for the
fractional wave operator immediately translate into assertions about
the Dirichlet datum of the solution and its associated Neumann
condition, and viceversa. Second, it is worth emphasizing that this relation can be
generalized to more general asymptotically AdS metrics, as we will do
it Section~\ref{section:geometry}. A third comment is that
Theorem~\ref{T.main} implies for
$\al\in(1,\frac n2)$ the operator $\square^\al$ can be naturally
interpreted as the scattering operator of the manifold. In Riemannian
signature, this connection is discussed in detail in~\cite{Mar}, and
quite remarkably
the interest in the
the conformally covariant operators on the boundary that it defines
(see e.g.\ \cite{Graham-Zworski,Mazzeo-Melrose,Fefferman-Graham}) was
originally fueled by  work of Newman, Penrose and LeBrun \cite{LeBrun}
on gravitational physics quite in the spirit of Maldacena's AdS/CFT
correspondence. 

Recall that, in
the Lorentzian case, the
construction of the scattering operator for a general asymptotically AdS
metric was carried out in~\cite{Vasy}, but the resulting operator was not characterized.

The paper is organized as follows. In Section~\ref{section:powers} we
will recall some basic facts and definitions about fractional wave
operators in a form that is particularly convenient for our
purposes, and in particular, a convolution formula with a singular kernel. In Section~\ref{section:extension} we will analyze the Klein--Gordon
equation and establish the main result in the half-space region of
anti-de Sitter space. This identity will be extended to more general
asymptotically AdS spaces, and to the global AdS space, in Section~\ref{section:geometry}.

\section{Fractional powers of the wave operator} \label{section:powers}

In this section we will recall some
basic facts about fractional powers of the wave operator, as defined
in~\eqref{Boxal}. In the following we will see
that, just as in the case of the fractional Laplacian ~\cite{Landkoff}, one can
represent $\Box^\al f$ as a regularized integral depending
analytically on the power $\al$ when $f$ is a sufficiently smooth
function. To derive this result in full generality, we will need the
analytic continuation of the classical hyperbolic Riesz potential
\cite{Riesz}, which will be explicitly constructed when the power is
not a positive half-integer.

A careful analysis of the poles of the multiplier~\eqref{symbol1}
shows that one can regard $\sigma_\al$ as a tempered distribution on
$\mathbb{R}^{n}$, analytic in the parameter $\al$ for $\al\neq
-\tfrac{n}{2}-k$ with $k$ a non-negative integer
(cf. e.g. \cite[Chapter III]{Gelfand}). This property is of crucial
importance to recover the fractional wave operator as a convolution
\[
	\Box^\al f=k_\al * f,
\]
where the kernel $k_\al$ coincides with the inverse Fourier transform
of $\sigma_\al$ in a sense that will be specified later on. Notice
that, while this relation holds true distributionally, it is not easy
to transform it into a pointwise converging formula due to the
singularities that the multiplier presents on the light cone.

To regularize the integrals that appear, we will use Riesz
distributions. Let us recall that, for any complex parameter $\al$
with $\Real\al>\frac{n}{2}-1$, $R_\al$  is the distribution whose action on a function $\vp\in
C^\infty_0(\RR^n)$ is given by the absolutely convergent integral
\begin{equation*}
	\langle R_\al, \varphi\rangle=C_{n,\al} \int_{\mathcal{K}_+^+} (t^2-|x|^2)^{\al-\frac{n}{2}} \varphi(t,x) \,dt\,dx,\quad \varphi\in C^\infty_0(\mathcal{K}_+^+),
\end{equation*}
where we write the points in $\RR^n$ as $(t,x)\in\RR\times\RR^{n-1}$,
\[
	\mathcal{K}_+^+:=\{(t,x)\in\mathbb{R}^{n}: t^2\geq |x|^2,\ t\geq 0\}.
\]
is the forward causal cone and the constant $C_{n,\al}$ takes the value
\begin{equation}
	 C_{n,\al}:=\frac{2^{1-2\al}\pi^{1-\frac{n}{2}}}{\Ga(\al)\,\Ga\big(\al+1-\frac{n}{2}\big)}.\label{Const}
\end{equation}

A straightforward computation shows that the map $\al\mapsto \langle
R_\al, \varphi\rangle$ is analytic in the half-plane
$\Real\al>\frac{n}{2}-1$. It is well-known (see e.g.~\cite{Kolk}) that
this mapping can be analytically continued to the whole complex plane
by means of the identity $\Box R_{\al+1}=R_\al$. Hence for
any complex number $\al$ it makes sense to consider the convolution of
the distribution $R_\al$ a smooth compactly supported function $f\in
C^\infty_0(\RR^n)$, which we will denote by $I_\al f$ and call the
{\em hyperbolic Riesz potential}\/ of~$f$. Notice that for $\Real\al>\frac{n}{2}-1$ the Riesz potential simply reads as
\begin{equation}
	I_\al f (t,x)=\int_{\mathcal{K}_+^+} (s^2-|y|^2)^{\al-\frac{n}{2}} f(t-s,x-y) \,ds\,dy.\label{RieszP}
\end{equation}

In the following proposition we establish the relationship among the powers of
the wave operator, the Riesz distribution and its associated
potential, showing that, as in the Euclidean case,
it is possible to understand $\Box^\al f$ as a regularized
integral represented by the analytic extension of the Riesz
potential. This relation was essentially stated in~\cite{Samko}
using the results of~\cite{Gelfand} on
the Fourier transform of analytically continued quadratic forms, but
we prefer to sketch the proof here rather than to refer to vague
variations on results
from the above references.

\begin{proposition}\label{P.sial}
Let $\al$ be a complex number such that $\al+\frac n2$ is not a
non-positive integer. Then the Fourier transform of $R_{\al}$ is the function
\begin{equation}
	\widehat{R}_{\al}(\tau,\xi)=\sigma_{-\al}(\tau,\xi),\label{symbol+Riesz}
\end{equation}
so, in particular, for any function $f\in C^\infty_0(\RR^n)$ the
$\al\th$ power of the wave operator is given by
\begin{equation}
\Box^\al f=I_{-\al} f.\label{PowersWave}
\end{equation}
\end{proposition}

\begin{proof}
Let us first assume
that $\Real\al>\frac{n}{2}-1$. The Fourier transform
$\widehat{R}_{\al}$ is then the distribution given by
\begin{align*}
\langle \widehat{R}_\al, \varphi\rangle&=\langle R_\al, \widehat{\varphi}\rangle=C_{n,\al}\int_{\mathcal{K}_+^+} (t^2-|x|^2)^{\al-\frac{n}{2}} \widehat{\varphi}(t,x) \,dt\,dx\\
&=C_{n,\al}\lim_{\epsilon\to 0^+}\int_{\mathcal{K}_+^+} (t^2-|x|^2)^{\al-\frac{n}{2}} e^{-\epsilon t}\widehat{\varphi}(t,x) \,dt\,dx\\
	&=C_{n,\al}\int_{\mathbb{R}^n} \lim_{\epsilon\to 0^+}\left( \int_{\mathcal{K}_+^+} (t^2-|x|^2)^{\al-				\frac{n}{2}}e^{-i\xi x} e^{-t(\epsilon+i\tau)}\,dt\,dx\right)\varphi(\tau,\xi) \,d\tau\,d\xi,
\end{align*}
where we have used dominated convergence and Fubini's theorem.

The inner integral is computed by passing to polar coordinates (for $n>3$; the case $n=3$ is much simpler) and using the well-known representation formulas \cite{Watson} for the Bessel functions
\begin{align}
	J_\nu(z) &= \frac{ (\frac{z}{2})^\nu }{ \al(\nu + \frac{1}{2} ) \sqrt{\pi} } \int_{-1}^{1} e^{izs}					(1 - s^2)^{\nu - \frac{1}{2} } \,ds, \qquad \Real(\nu)>-\tfrac{1}{2},\ z\in\mathbb{C}\\
	K_{\nu}(z)&=\frac{\sqrt{\pi}(\frac{z}{2})^{\nu}}{\al(\nu+\frac{1}{2})}\int_{1}^{\infty}e^{-					zw}(w^{2}-1)^{\nu-\frac{1}{2}}\,dw,\qquad \Real(\nu)>-\tfrac{1}{2},\ |\Arg(z)|<\tfrac{\pi}{2},\label{Bessel}
\end{align}
and the Bessel integral
\[
	\int_{0}^{\infty}r^{\mu+\nu+1}K_{\mu}(a r) J_{\nu}(br)dr=\frac{(2a)^{\mu}(2b)^{\nu}\Ga(\mu+
		\nu+1)}{(a^{2}+b^{2})^{\mu+\nu+1}},
\]
for $\Real\nu+1>|\Real\mu|$ and $ \Real a>\Imag b$. It is not difficult then to check that one can write our integral in the half plane $\Real\al>\frac{n}{2}-1$ as
\begin{align*}
	C_{n,\al}\int_{\mathcal{K}_+^+} (t^2&-|x|^2)^{\al-\frac{n}{2}}e^{-i\xi x} e^{-t(\epsilon+i\tau)} \,dt\,dx\\
	&=2^{\al}\pi^{\frac{n}{2}-1}\al(\al+1-\tfrac{n}{2})C_{n,\al}\\
	&\qquad\times |\xi|^{\frac{3-n}{2}}(\epsilon+i\tau)^{\frac{n-1}{2}-\al}\int_{0}^\infty
	 r^\al K_{\al+\frac{1-n}{2}}(r(\epsilon+i\tau))J_{\frac{n-3}{2}}(r|\xi|)\,dr\\
	&=(|\xi|^2+(\epsilon+i\tau)^2)^{-\al}.
\end{align*}
Taking now the limit as $\epsilon\to 0^+$ we get \eqref{symbol+Riesz}.

To complete the proof we recall that, as pointed out right before
introducing the concept of Riesz distribution, $\sigma_{\al}$ defines
a tempered distribution that is analytic for any complex $\al$ except
for $\al=-\tfrac{n}{2}-k$ with $k\in\mathbb{N}$. Therefore, although
we have proved $\eqref{symbol+Riesz}$ when the parameter takes values
in a certain open set of the complex plane, this relation must hold in
the whole domain where $R_\al$ can be analytically
continued. Moreover, since by definition $\widehat{\Box^\al
  f}=\sigma_\al \widehat{f}$ for all $\al$, it is also true that $\Box^\al
f=R_{-\al}* f=I_{-\al} f$. The proposition then follows.
\end{proof}

The above proposition is roughly the analog of the formula~\eqref{Dega2} for the
fractional Laplacian, which for $0<\al<n$ allows us to write the
fractional Laplacian $(-\Delta)^{-\al}f$, up to a multiplicative
constant, as the convolution of $f$ with the locally integrable
function $|x|^{2\al-n}$. Notice, however, that the singularities in
the integral kernel are much stronger here.

For the benefit of the reader we shall next record an explicit formula for~$\Box^\al f$ in terms of
integrals regularized via suitable finite difference operators, which we borrow from \cite[Eq.~(9.93)]{Samko},
and connect it with the previous proposition. Before we can state the
result, let us first introduce some further notation. For $q$ a real
parameter and $k,l\in\mathbb{N}$, let us denote the $q$-number of $k$,
its $q$-factorial and the $q$-binomial coefficient as
\[
    [k]_{q}={\frac {1-q^{k}}{1-q}},\qquad  [k]_{q}!=[1]_{q}[2]_{q}\cdots [k]_{q},\qquad  \binom{l}{k}_{q}={\frac {[l]_{q}!}{[k]_{q}!\,[l-k]_{q}!}}\quad (k\leq l),
\]
respectively. In addition, let us define the $q$-functions
\begin{equation}
C_{k}^l=q^{k(\frac{k+1}{2}-l)}\binom{l}{k}_q,\qquad A^l_{\mu}=\sum _{k=0}^l (-1)^{k} q^{k\mu} C_{k}^l,\label{qfunc}
\end{equation}
with $\mu\in\mathbb{C}$ and where we omit the dependence on $q$ for notational simplicity.

We are now ready to write down the integral formula for $\Box^\al f$.
For simplicity we will assume that $f\in C^\infty_0(\RR^n)$, but
the result it is still true e.g.\ for $C^l(\RR^n)$ functions
that decay fast enough at infinity, with $l>2\al$.

\begin{proposition}\label{P.conv}
Let us take a real $\al\in (0,\frac{l}{2})$, where $l\in\mathbb{N}$
and we assume that $\al$ is not a half-integer. Then for any $f\in C^\infty_0(\RR^n)$ the fractional
wave operator $\Box^\al$ is given by the absolutely
convergent integral
\begin{equation}
	\Box^{\al}f(t,x)=
		C_{n,-\al}\int_{\mathbb{R}_+\times\mathbb{R}^{n-1}}\dfrac{\Delta_{s,y}^{l,\al}f(t,x)}{s^{\frac{n}{2}+\al}
			|{y}|^{n+2\al-1}}ds\ dy,\label{WaveIntegral}
\end{equation}
where $ \Delta_{s,y}^{n,\al}$ stands for the difference operator
\[
	 \Delta_{s,y}^{l,\al}f(t,x)=\dfrac{1}{A^{l^*}_{\frac{n}{2}-1+\al}A^l_{2\al} }
	\sum_{j=0}^{l^*}\sum_{k=0}^l (-1)^{j+k}C_{j}^{l^*} C_{k}^l
			\dfrac{(1+q^j s)^{2\al}}{(2+q^j s)^{\frac{n}{2}+\al}}f\Big( t-q^k |{y}|, {x}-\dfrac{q^k {y}}{1+q^j s} \Big),
\]
$l^*$ is the integer part of
$\frac{n+l-1}{2}$, $q\neq1$ is a positive constant, $C_{k}^l$ and
$A^l_{\mu}$ are the $q$-functions defined in \eqref{qfunc}, and the
constant $C_{n,\al}$ is given by~\eqref{Const}.
\end{proposition}

\begin{proof}

By Proposition~\ref{P.sial} the result is equivalent to showing
that the integral \eqref{WaveIntegral} represents the extension of
$I_{-\al}f$ to the strip $0<\al<l/2$. Firstly, observe that this
integral converges absolutely on this interval except for integers and
semi-integers values of $\al$. Indeed, a tedious but straightforward
computation of the Taylor series of $\Delta_{s,y}^{l,\al}f$ about zero
combined with the fact that $A^l_{m}=0$ for $0\leq m\leq l-1$ as a
consequence of the well-known identity \cite{Cigler}
\begin{equation*}
\sum_{k=0}^l q^{k\frac{(k-1)}{2}}z^k \binom{l}{k}_q=\prod_{k=0}^{l-1}(1+z q^k)
\end{equation*}
applied to $z=-q^{1+m-l}$, reveal that
\begin{align*}
	|\Delta_{s,y}^{l,\al}f|\leq C_q  s^{l^*}|{y}|^l  \| D^l f \|_{\infty}+\mathcal{O}\big(s^{l^*}|{y}|^{l+1} \| D^{(l+1)} f\|_{\infty}\big)
\end{align*}
for $(s,y)$ close to zero and where $C_q$ is a real constant that depends only on $q$. From this bound and since $f$ vanishes at infinity, it immediately follows that the integral is finite when $\al\in (0,\frac{l}{2})$ except for the zeros of $A^{l^*}_{\frac{n}{2}-1+\al}$ and $A^l_{2\al}$, which correspond to the points of the form $\al=\frac{k}{2},\ k\in\mathbb{N}$ when $q\neq 1$ is a positive real number.

Now we need to prove that the formula above actually represents the analytic continuation of the Riesz potential from the half plane $\Real\al>\frac{n}{2}-1$ to the range $\al\in (0,\tfrac{l}{2})$. Note that by \eqref{PowersWave} and the form of the differences operator, it suffices to show that \eqref{WaveIntegral} coincides with the expression of the potential given in \eqref{RieszP} replacing $\al$ by $-\al$. This assertion can be readily checked by introducing new variables
\[
u:=q^{k} |{y}|,\qquad v:=q^{k} \dfrac{|{y}|}{1+q^j s}
\]
and changing to polar coordinates in ${y}$. Incidentally, since the
above substitutions also remove the dependence of the operator
$\Box^\al f$ on the parameter $q$, then one can safely choose any
positive $q$ other than the limit value $q=1$.
\end{proof}

\begin{remark}
It is worth noticing that even in the simplest case, $0<\al<1$ and
$n=2$, the formula for the fractional wave operator is much more
involved than its Euclidean counterpart and cannot be deduced from
it. This reflects the different nature of the singularities of the
corresponding kernels: the entire light cone in the hyperbolic case
and a single point in the Euclidean setting. Remarkably, in the simple
case that we are now discussing ($n=2$, $0<\al<1$) there is another realization of $\Box^\al$
(cf.\ \cite[Theorem~9.30]{Samko}) that is easier to compare with its fractional Laplacian analog:
\begin{equation}
	\Box^\al f (t,x)=\dfrac{C_{2,-\al }}{2^{1+2\al}}\int_{\mathcal{K}_+^+} \dfrac{T_{s,y}f(t,x)}{(s^2-y^2)^{1+\al}} ds\ dy ,\label{kernel2}
\end{equation}
where we are using the finite difference operator
\begin{align*}
	T_{s,y}f(t,x):= f(t,x)-f\big(t-\tfrac{s+y}{2},x-\tfrac{s+y}{2}\big)- f\big(t-\tfrac{s-y}{2},x+\tfrac{s-y}{2}\big) +f(t-s,x-y).
\end{align*}
Note that one can readily use this formula
and the crude approximations
\[
T_{s,y}f(t,x) = -s y\ \Box f(t,x)+O(s^2+y^2)\,,\qquad \Gamma(-1+\ep)^{-1}=-\ep+O(\ep^2)
\]
to prove
the pointwise convergence of $\Box^\al f(t,x)$ to $\Box f(t,x)$ as
$\al\to 1$.
\end{remark}

\section{The Klein--Gordon equation in AdS spaces}\label{section:extension}

In this section we will connect the fractional powers of the wave
operator with the solutions to the mixed initial-boundary problem
corresponding to the Klein--Gordon equation in an anti-de Sitter
space. Our specific goal here is to give a local realization of the
fractional wave operator as a Dirichlet-to-Neumann map that is the
Lorentzian counterpart of the relation for the fractional Laplacian derived in~\cite{Caffarelli}
and extended in~\cite{Mar}. For this
purpose, we will use a Laplace--Fourier transform in order to
transform our wave equation into an ODE that contains the relevant
information about the solution at infinity,
which will enable us to derive the Dirichlet-to-Neumann map in Fourier space.

Recall that our starting point was
the Klein--Gordon equation
\[
	\Box_{g^+}\phi+\Big(\al^2-\dfrac{n^2}{4}\Big)\phi=0,
\]
where $\Box_{g^+}$ denotes
the wave operator associated to the AdS metric \eqref{metric1}.
As discussed in the Introduction,
this equation can be thought as a wave
equation with coefficients that are singular at conformal infinity
and whose indicial roots are $n/2\pm \al$.
For simplicity, we will henceforth assume that $\al$ is not a
half-integer (i.e., $2\al\not\in \NN$) in order to ensure that the solution
does not have a logarithmic branch cut. The argument carries over to
the case that $\al\not\in\NN$ with minor modifications.

Let us begin with the analysis of the
wave equation ~\eqref{eqphi2}.
For this, is convenient to define the function
\[
u(t,x,y):=y^{\al-\tfrac{n}{2}}\phi(t,x,y)\,,
\]
which satisfies the equation
\begin{equation}\label{equ2}
\partial_{tt} u-\De_xu-\partial_{yy }u-\frac{1-2\al}y \partial_y u=0
\end{equation}
with Dirichlet datum
\begin{equation}\label{Dirichletu}
u(t,x,0)=f(t,x)\,.
\end{equation}
We shall next prove a bound for~$u$ that will be needed later in this
section. To state it, let us consider the weighted Lebesgue space
\[
\La:= L^2(\RR^{n}_+,y^{1-2\al} dx \,dy)\,,
\]
endowed with the norm
\[
\|v\|_{\La}^2:=\int_{\RR^{n}_+}v^2\, y^{1-2\al} dx \,dy\,,
\]
and denote by $\HH$ (respectively, $\HHo$) the closure of
$C^\infty_0(\overline{\RR^n_+})$ (respectively, $C^\infty_0(\RR^{n}_+)$) with
respect to the norm
\[
\|v\|_{\HH}^2:=\int_{\RR^{n}_+}\big(|\nabla_xv|^2+(\partial_y v)^2\big)\, y^{1-2\al} dx \,dy\,.
\]

In the following theorem we only state a qualitative result for
boundary data $f\in C^\infty_0(\RR^n)$, which is what we need in the paper, but in the
proof we provide quantitative estimates that obviously extend the
result to data in more general function spaces.

\begin{lemma}\label{L.existence}
Let $k$ be the lowest integer such that $k>\frac{1+\al}{2}$.
Given a boundary datum $f\in C^\infty_0(\RR^n)$ and an integer $j<k$,
define $$u_j(t,x,y):=y^{2j}\,\chi(y)\,\Box^j f(t,x) \,,$$
where $\chi(y)$ is a fixed smooth cutoff function identically $1$ in $y<1$
and $0$ in $y>2$.
Then there are real numbers $c_j$ and a function $v\in
L^\infty(\RR,\dot H^1_\al)$ such that
\begin{equation}
	u(t,x,y):=\sum_{j=0}^{k-1}c_j u_j(t,x,y) + v(t,x,y)\,, \label{eqsum}
\end{equation}
is the unique solution of Equation~\eqref{equ2}
with trivial initial data $u(-\infty,x,y)=u_t(-\infty,x,y)=0$
and boundary condition $u(t,x,0)=f(t,x)$.
\end{lemma}
\begin{proof}
Let us consider the initial condition
\begin{equation}\label{inicu}
u(t_0,\cdot)=u_t(t_0,\cdot)=0\,,
\end{equation}
where $t_0$ is any number such that
$f(t,x)=0$ for all $t<t_0$. We shall see in the proof that the
solution is independent of the choice of $t_0$, so it is equivalent to
imposing $u(-\infty,\cdot)=u_t(-\infty,\cdot)=0$.

We shall start by considering an auxiliary non-homogeneous Cauchy
problem of the form
\begin{subequations}
\begin{align}
	\partial_{tt} v-\De_xv-\partial_{yy} v-\frac{1-2\al}y \partial_y v&=F(t,x,y), \\
	v(t_0,\cdot)=v_t(t_0,\cdot)&=0.
\end{align}\label{wave*}
\end{subequations}
We shall next show that if $F\in L^1(\RR,\La)$, there is a unique
solution
$$v\in L^2\loc(\RR, \HHo)\cap H^1\loc(\RR,
\La)$$
to this equation. Notice that the fact that $v(t,\cdot)$ takes
values in $\HHo$ means that we are imposing the boundary condition
$v|_{y=0}=0$.

To prove this, we will use an a priori estimate for the energy
associated to the solution $v$, which we define as
\[
	E_v(t):=\dfrac{1}{2}\int_{\mathbb{R}^{n+1}_+} (v_t^2+|\nabla_x
        v|^2+v_y^2)\,y^{1-2\al} \,dx\,
        dy.
\]
To prove this estimate, it is standard that by a density argument one can assume that $v$ is
in $C^\infty_0(\RR^{n+1}_+)$, differentiate under the integral sign
and integrate by parts to find that
\begin{align*}
	\frac d{dt} E_v(t)&=\int_{\RR^n_+} (v_t v_{tt}+\nabla_x v_t
        \cdot \nabla_x v+ v_yv_{yt})\, y^{1-2\al} dx\,dy \\
		&=\int_{\RR^n_+}
                v_t\left(v_{tt}-\De_xv-v_{yy}-\frac{1-2\al}y
                  v_y\right)  \, y^{1-2\al} dx\, dy\\
&=\int_{\RR^n_+} F\, v_t\, y^{1-2\al} dx\, dy\\
&\leq
C\|F(t,\cdot)\|_{\La}\, E_v(t)^{1/2}\,.
\end{align*}
Using now Gr\"{o}nwall's inequality, we arrive at
\[
	E_v(t)^{1/2}\leq   E_v(t_0)^{1/2}+C\bigg|\int_{t_0}^t \|F(t',\cdot)\|_{\La} dt'\bigg|,
\]
which, by the trivial initial conditions, readily implies the estimate
\begin{equation}
\sup_{t\in\RR}\,(\| v(t,\cdot)\|_{\HH}+\|v_t(t,\cdot)\|_{\La}) \leq C
\int_{-\infty}^{\infty}\|F(t',\cdot)\|_{\La}\, dt',
\label{estimate*}
\end{equation}
thereby ensuring that $v\in L^\infty(\RR,\HHo)$.
It is standard that this estimate leads to the existence of a unique solution
$v\in L^2\loc (\RR,\HHo)\cap H^1\loc(\RR, \La)$ to the
problem~\eqref{wave*}.

To apply the estimate~\eqref{estimate*} in our problem,
let us set
\[
v(t,x,y):= u(t,x,y)-\sum_{j=0}^{k-1}\frac{(-1)^j
  y^{2j}\chi(y)}{\prod_{l=1}^j 4l(l-\al)}\Box^j f(t,x)\,,
\]
where the product is to be taken as~1 when $l=0$.
Using now that
\[
\Big(\pd_{yy}+\frac{1-2\al}y\pd_y\Big)y^j=j(j-2\al)\, y^{j-2}\,,
\]
a direct calculation shows that
\begin{equation}\label{eqv}
	\partial_{tt} v-\De_xv-\partial_{yy}v-\frac{1-2\al}y
        \partial_y v=\frac{(-1)^{k}y^{2k-2}\chi(y)}{\prod_{l=1}^{k-1} 4l(l-\al)}\Box^{k}f+
        \sum_{j=0}^{k-1}\chi_j(y)\, \Box^j f\,,
      \end{equation}
where $\chi_j (y)$ is a smooth function whose support is contained in
the interval $[1,2]$. Moreover, by construction $v$ satisfies the
initial and boundary conditions
\[
v(t_0,x,y)=v_t(t_0,x,y)=v(t,x,0) =0.
\]
The point now is that, as the right hand side of~\eqref{eqv} behaves
as $y^{2k-2}$ and vanishes for $y>2$, it is easy to see that it is in
$L^1(\RR,\La)$, so the estimate~\eqref{estimate*} ensures
that $u$ written as in \eqref{eqsum}
is the unique solution of \eqref{equ2}
satisfying the boundary condition $u|_{y=0}=f$
and vanishing initial data.
\end{proof}
\begin{remark}
Arguing as in~\cite[Proposition 5.2]{Enciso}, we could have proved higher regularity
estimates for the solution in suitable weighted spaces, but we will
not need that result. The global $L^\infty$~bound in time, on the contrary,
will be essential and is not proved in the aforementioned paper, as it does not
hold for the more general equations there considered.
\end{remark}

Given a positive real parameter $\al$ that is not a half-integer, let us write it as
$\al=\al_0+m$, where $m$ denotes its integer part and $\al_0\in (0,1)$.
The generalized Dirichlet-to-Neumann map $\Lambda_\al$ is then defined as in~\cite[Theorem 3.3]{Mar}
\begin{equation}
	\Lambda_\al f(t,x)=c_\al\lim\limits_{y\to 0^+}
        y^{2(1-\al_0)}\Big(\frac{1}{y}\, \pd_y\Big)^{m+1} u(t,x,y),\label{DNgen}
\end{equation}
with $$c_\al:=(-1)^{m+1} 2^{\al+\al_0-1}\frac{\Ga(\al)}{\Ga(1-\al_0)}$$ an
inessential normalizing factor.

The following theorem, which is the central result of this paper, is
essentially a rewording of Theorem~\ref{T.main}:

\begin{thm}\label{T.general}
  Given a positive real number $\al$ that is not an integer and
  $f\in C^\infty_0(\RR^n)$, let $u$ be the solution of the
  initial-boundary problem in the Poincar\'e-AdS space~\eqref{equ2} with
  Dirichlet datum~\eqref{Dirichletu} and initial data
  $u(-\infty,\cdot)=u_t(-\infty,\cdot)=0$.  Then the map $\Lambda_\al$
  defined in~\eqref{DNgen} is given by the $\al\th$ power of the wave
  operator:
\begin{equation}
\Lambda_\al f=\Box^\al f.
\end{equation}
\end{thm}

\begin{proof}
To begin with, let us start by noticing that the solution with the
initial condition $u(-\infty,\cdot)=u_t(-\infty,\cdot)=0$ is well
defined, and it can be equivalently defined by the condition
$u(t_0,\cdot)=u_t(t_0,\cdot)=0$ where $t_0$ is any number such that
$f(t,x)=0$ for all $t<t_0$. In what follows, $t_0$ will denote a
number with this property.

Consider the Laplace transform of $u(t,\cdot)$,
\[
	U(s,\cdot)=\int_{t_0}^\infty e^{-s(t-t_0)}\ u(t,\cdot)\ dt,
\]
where $s=\epsilon+i\tau$ with $\epsilon>0$ and $\tau\in\mathbb{R}$.
Notice this expression converges as a vector-valued function because
$u$ is a Banach-space valued $L^\infty$~function of time by
Lemma~\ref{L.existence}. (When $\al<1$, the norm is
simply that of $\dot H^1_\al$, while for $\al\geq1$ the norm must also control the
terms $u_j$ appearing in~\eqref{eqsum}, for example by decomposing
\[
u(t,x,y)=:\chi(y)\sum_{j=0}^{k-1}y^{2j}\, \bar u_j(t,x)+ v(t,x,y)
\]
and adding the $L^\infty$ norm of the
functions $\bar u_j$ and the $L^\infty(\RR,\dot H^1_\al)$ norm of~$v$).

Furthermore, observe that one can then recover $u$ through the inverse Laplace transform formula
\[
	u(t,\cdot)=\dfrac{1}{2\pi i}\int_{\epsilon-i\infty}^{\epsilon+i\infty} e^{s(t-t_0)}\ U(s,\cdot) \ ds,
\]
where the integration contour is the vertical line of numbers whose
real part is $\epsilon>0$.

We tackle the problem of finding explicit solutions to problem \eqref{equ2}-\eqref{Dirichletu} as
follows. First, we apply the Laplace transform on the equation for~$u$ to remove the time derivatives by integration by parts and then use the trivial initial conditions in the LHS above,
\begin{align*}
	\int_{t_0}^\infty e^{-s(t-t_0)}\ {u}_{tt}(t,\cdot)\ dt
		=s^2 {U}(s,\cdot)- {u}_t(t_0,\cdot)-s {u}(t_0,\cdot)=s^2 {U}(s,\cdot)
\end{align*}
to find that $U(s,x,y)$ satisfies the equation
\begin{equation*}
	{\partial}_{yy}U(s,x,y)+\dfrac{1-2\al}{y}{\partial}_y U(s,\xi,y)+(\De_x-s^2){U}(s,x,y)=0.
\end{equation*}

Next we take the Fourier transform in space
with respect to the variable $x$, which here is denoted with a tilde
to avoid confusions with the space-time Fourier transform. This yields
the ODE
\begin{equation*}
	\partial_{yy}\widetilde{U}(s,\xi,y)
+\dfrac{1-2\al}{y}\partial_y\widetilde{U}(s,\xi,y)-(|\xi|^2+s^2)\widetilde{U}(s,\xi,y)=0.
\end{equation*}
The general solution of this equation can be written as a linear
combination of Bessel functions multiplied by a certain power of
$y$, and spans a two-dimensional vector space. However, by Lemma~\ref{L.existence} the solution $u(t,\cdot)$
is given by the sum of a function bounded
in~$\HH$ and other terms that are uniformly bounded in~$y$,
so we may discard the independent solution
that grows exponentially in~$y$ as $y\to\infty$ to arrive at the formula
\begin{equation}\label{explicit-formula}
	\widetilde{U}(s,\xi,y)=\dfrac{2^{1-\al}}{\Gamma(\al)}y^\al (|\xi|^2+s^2)^{\frac{\al}{2}} K_{\al}(y\sqrt{|\xi|^2+s^2}) \widetilde{F}(s,\xi),
\end{equation}
where $K_{\al}$ denotes the modified Bessel function of the second
kind defined in \eqref{Bessel}. The constant (which depends on $s$ and
$\xi$) has been chosen to ensure that the Dirichlet condition is
satisfied:
\[
\widetilde{U}(s,\xi,0)= \widetilde{F}(s,\xi).
\]
Here and in what follows, $F(s,x)$ denotes the Laplace transform
of~$f(t,x)$, computed as above, and the tilde denotes the Fourier transform in space.

Therefore, using the identity
\[
	\left(\frac{1}{z}\frac{d}{dz}\right)^{k}(z^{\nu}K_{\nu}(z))=(-1)^k z^{\nu-k}K_{\nu-k}(z)
\]
for modified Bessel functions, we readily infer that
the Dirichlet-to-Neumann map \eqref{DNgen} reads in the Laplace-Fourier space as
\[
	\widetilde{\Lambda_\al F}(s,\xi):=c_\al \lim\limits_{y\to 0^+} y^{2(1-\al_0)}\Big(\frac{1}{y}\pd_y\Big)^{m+1} \widetilde{U}(s,\xi,y)=(|\xi|^2+s^2)^\al \widetilde{F}(s,\xi).
\]

The key point now is that one can relate the inverse Laplace transform with the inverse Fourier transform in time using the freedom in the choice of the parameter~$\epsilon$. To show this, notice that by the Laplace inversion formula, we have that the Dirichlet-to-Neumann map in the variables $(t,\xi)$ can be written as
\begin{align*}
	\widetilde{\Lambda_\al f} (t,\xi)
	&=\dfrac{1}{2\pi i} \int_{\epsilon-i\infty}^{\epsilon+i\infty}  e^{s(t-t_0)} (|\xi|^2+s^2)^\al
	\widetilde{F}(s,\xi) \ ds \\
	&=\dfrac{1}{2\pi} \int_{-\infty}^{\infty}  \int_{t_0}^{\infty} e^{(\epsilon+i\tau)(t-t')} (|\xi|^2+(\epsilon+i\tau)^2)^\al \widetilde{f}(t',\xi) \ d\tau\ dt'.
\end{align*}
Since the latter integral does not depend on the value of $\epsilon$ and $f$ is smooth with compact support in $\{ t\geq t_0\}$, by the dominated convergence theorem one can take the limit as $\epsilon\to 0^+$ inside the integral to find that
\begin{equation*}
	\widetilde{\Lambda_\al f} (t,\xi)
	= \dfrac{1}{2\pi } \int_{-\infty}^{\infty} e^{i\tau t} \sigma_\al(\tau,\xi) \widehat{f}(\tau,\xi) \ d\tau ,
\end{equation*}
with $\sigma_\al$ as in \eqref{symbol2} and $\widehat{f}(\tau,\xi)$ denoting the Fourier transform of $ f$ with respect to both time and space variables as in Section 2. Taking now the Fourier transform with respect to the time, we obtain
\[
	\widehat{\Lambda_\al f}(\tau,\xi)= \sigma_\al(\tau,\xi) \widehat{f}(\tau,\xi),
\]
proving our claim.
\end{proof}

A consequence of the proof is an explicit formula for the spacetime energy of
the solution in terms of its boundary datum that is analogous to the
result in Euclidean signature from \cite{Caffarelli}:

\begin{corollary}
With $u$ and $f$ be as in Theorem~\ref{T.general}, the total energy of
$u$ is
\[
\int_{\RR^{n+1}_+}\big[(\partial_y u)^2 +|\nabla_x u|^2+(\partial_t
u)^2\big]\,y^{1-2\al} \,dt\,dx\,dy =
C_\al\int_{\RR^n}\si_\al(\tau,\xi)\, |\widehat f(\tau,\xi)|^2\, d\tau\, d\xi\,,
\]
with $C_\al$ a nonzero constant.
\end{corollary}
\begin{proof}
We first note that the Laplace transform of the function $u(s,\cdot)$ at $s=\epsilon+i\tau$,
$$U(\epsilon+i\tau)=\int_0^\infty e^{-i\tau t} e^{-\epsilon t} u(t,\cdot)\,dt,$$
is the Fourier transform of the function $e^{-\epsilon t} u(t,\cdot)$. Thus by Plancherel theorem we can write
\begin{equation}\label{Plancherel}
\int_{\epsilon-i\infty}^{\epsilon+i\infty} |U(s,\cdot)|^2\,ds=\int_{-\infty}^{+\infty} |U(\epsilon+i\tau,\cdot)|^2\,d\tau=\int_0^\infty e^{-2\epsilon t}|u(t,\cdot)|^2\,dt.
\end{equation}

Now we consider the energy for equation \eqref{equ2}, given by
\begin{equation*}\begin{split}
E&=\int_{0}^\infty y^{1-2\al} \int_{\mathbb R^{n-1}}\int_0^\infty \big[(\partial_y u)^2 +|\nabla_x u|^2+(\partial_t u)^2\big]\,dt\,dx\,dy\\
&=\lim_{\epsilon\to 0} \int_{0}^\infty y^{1-2\al} \int_{\mathbb R^{n-1}}\int_0^\infty \big[(\partial_y u)^2 +|\nabla_x u|^2+(\partial_t u)^2\big]e^{-2\epsilon t}\,dt\,dx\,dy,
\end{split}
\end{equation*}
which, using Plancherel identity \eqref{Plancherel} for the Laplace transform, becomes
\begin{multline*}
E=\lim_{\epsilon\to 0} \int_{0}^\infty y^{1-2\al} \int_{\mathbb R^{n-1}}\int_{\epsilon-i\infty}^{\epsilon+i\infty} \big[|\partial_y U(s,x,y)|^2 +|\nabla_x U(s,x,y)|^2\\+s^2|U(s,x,y)|^2\big]\,ds\,dx\,dy.
\end{multline*}
Now we take Fourier transform in the variable $x$, yielding
\begin{multline*}
E=\lim_{\epsilon\to 0} \int_{0}^\infty y^{1-2\al} \int_{\mathbb R^{n-1}}\int_{\epsilon-i\infty}^{\epsilon+i\infty} \big[|\partial_y \widetilde U(s,\xi,y)|^2 +(|\xi|^2+s^2)|\widetilde U(s,\xi,y)|^2\big]\,ds\,d\xi\,dy.
\end{multline*}
Substituting the explicit expression \eqref{explicit-formula} we arrive at
\begin{multline*}
E=\lim_{\epsilon\to 0} \int_{0}^\infty y^{1-2\al} \int_{\mathbb R^{n-1}}\int_{\epsilon-i\infty}^{\epsilon+i\infty} (|\xi|^2+s^2)\big[|K'_1(y\sqrt{|\xi|^2+s^2})|^2\\+|K_1(y\sqrt{|\xi|^2+s^2})|^2 \big]|\widetilde F(s,\xi)|^2\,ds\,d\xi\,dy,
\end{multline*}
where, for simplicity, we have set $K_1(y):=\dfrac{2^{1-\al}}{\Gamma(\al)} y^\al  K_{\al}(y)$.
A change of variable allows us to integrate in~$y$, obtaining
\begin{equation*}
E=\lim_{\epsilon\to 0} C_\alpha\int_{\mathbb R^{n-1}}\int_{\epsilon-i\infty}^{\epsilon+i\infty} (|\xi|^2+s^2)^\alpha |\widetilde F(s,\xi)|^2\,ds\,d\xi
\end{equation*}
with $C_\al$ an explicit constant,
as claimed.
\end{proof}



\section{Application to other asymptotically AdS metrics}\label{section:geometry}

Our point in this section is to show that the identities established in Theorem~\ref{T.main} remain valid in
considerably more general situations. We will illustrate this fact by
connecting other fractional wave operators with the
Dirichlet-to-Neumann map (or, more generally, the scattering operator) of
two simple classes of static asymptotically AdS manifolds:

\subsubsection*{Fractional waves in product spaces}

Consider a compact Riemannian manifold $\mathcal M$ of dimension $n-1$
endowed with a Riemannian metric $g_0$ and take the natural wave
operator on $\RR\times\mathcal M$, which is
\[
	\Box_0:=\partial_{tt}-\Delta_{g_0},
\]
where $\Delta_{g_0}$ stands for the Laplace-Beltrami operator on
$\mathcal{M}$.

Since $\mathcal{M}$ is compact, we can take an
orthonormal basis $\{Y_j\}_{j\in\mathbb{N}}$ of eigenfunctions of the
Laplacian $\Delta_{g_0}$, which satisfy
\[
-\Delta_{g_0} Y_j=\lambda_j^2 Y_j\,,
\]
and write any $L^2$ function $f$ on $M$ (depending on~$t$ as a parameter) as the $L^2$-convergent series
$f(t,\cdot)=\sum_j f_j(t) Y_j(\cdot)$. Let us denote by $(\Box_0 f)_j
(t)$ the $j\th$ component of the function $\Box_0 f$ in this basis. Taking the Fourier transform with respect to $t$ we obtain that
\[
\widehat{\Box_0 f_j}(\tau)=(\lambda_j^2-\tau^2) \widehat{f}_j(\tau)
\]
and, given a real parameter $\al$, we can define here the $\al\th$ power of the wave operator as the pseudo-differential operator that in the Fourier space reads as
\begin{equation}
	\widehat{\Box_0^\al f_j}{(\tau)}:=\sigma_\al(\tau,\lambda_j) \widehat{f}_j(\tau),
\end{equation}
with $\sigma_\al$ the function defined in \eqref{symbol2}.

Consider now an
$(n+1)$-dimensional Lorentzian spacetime with the metric
\begin{equation}
g^+:=\frac{dt^2-dy^2-g_0}{y^2},\label{metric2}
\end{equation}
where $t\in\RR$ is the time coordinate and $y\in\mathbb{R}_+$
is a spatial coordinate.
Comparing with the metric defined in \eqref{metric1}, it is clear that the
Klein-Gordon equation with parameter $\mu:=(\alpha^2-n^2/4)$
associated to the metric \eqref{metric2} then takes the form
\[
\partial_{tt}\phi-\De_{g_0}\phi-\partial_{yy} \phi- \frac{1-n}{y}\partial_y\phi +\frac{4\al^2-n^2}{4y^2}\phi=0\,,
\]
where one must prescribe some suitable initial-boundary conditions for
the scalar field~$\phi$.
As in the last section, the above equation can be rewritten in terms of the rescaled function
$u:=y^{\al-\frac{n}{2}} \phi$ as
\begin{equation}
\partial_{tt} u=\Delta_{g_0} u+\frac{1-2\al}{y}\partial_y u+\partial_{yy} u,\label{wave2}
\end{equation}
where we take trivial initial data at time $-\infty$ and prescribe the
boundary condition at timelike conformal infinity: $u|_{y=0}=f$.

We can next define the (generalized) Dirichlet-to-Neumann map through its coefficients
\begin{equation}
(\Lambda_\al f)_j=c_\al \lim\limits_{y\to 0^+} y^{2(1-\al_0)}\Big(\dfrac{1}{y}\pd_y\Big)^{m+1} u_j,\label{DNgen2}
\end{equation}
with $c_\al$ as before, $\al=\al_0+m$, $m=\lfloor\al\rfloor$ the
integer part of~$\al$  and $\al_0\in (0,1)$.

By means of expansion in eigenfunctions of the Laplacian $\De_{g_0}$
and the Laplace transform in time (together with the vanishing initial conditions),
the equation \eqref{wave2} can be transformed into our well known ordinary equation
\[
	 \partial_{yy}{U}_j(s,y)+\dfrac{1-2\al}{y}\partial_y{U}_j(s,y)-(\lambda_j^2+s^2){U}_j(s,y)=0,
\]
with boundary condition $F_j(s)=U_j (s,0)$,
where $$U_j(s,y):=\int_{t_0}^\infty e^{-s(t-t_0)}\,u_j(s,y)\,dt\,$$
is the Laplace transform of the coefficient $u_j$
with $s=\epsilon+i\tau$ and $\ep$ a fixed positive constant.
Notice that $U_j$ can be shown to be well defined for $\epsilon>0$ by
an $L^\infty$~bound in time that goes exactly as in Lemma~\ref{L.existence}.

Arguing just as in the previous section,
we find that the Dirichlet-to-Neumann map in the transformed space reads as
\[
	\Lambda_\al F_j:= c_\al \lim\limits_{y\to 0^+} y^{2(1-\al_0)}\Big(\dfrac{1}{y}\pd_y\Big)^{m+1} {U}_j(s,y)=(\lambda_j^2+s^2)^\al {F}_j(s),
\]
and therefore, by the inverse Laplace transform formula
and the Fourier transform in time,
\[
	\widehat{\Lambda f}_j(\tau)=\sigma_\al(\tau,\lambda_j) \widehat{f}_j(\tau).
\]
Thereby, we can identify the Dirichlet-to-Neumann map
in the Lorenztian space with metric \eqref{metric2}
with the powers of the wave operator $\Box_0$.

\subsubsection*{The global anti-de Sitter space}

Consider now the global anti-de Sitter mentioned in the introduction,
which is diffeomorphic to $\RR^{n+1}$ and one can describe through spherical coordinates
\[
(t,r,\theta)\in \RR\times\RR_+\times\mathbb{S}^{n-2}\,,
\]
which cover the whole manifold modulo the usual abuse of notation at
the origin. In these coordinates the metric of $\mathrm{AdS}_{n+1}$ reads as
\begin{equation}
	g^+:=(1+r^2)dt^2-\dfrac{1}{1+r^2}dr^2 -r^2 g_{\mathbb{S}^{n-2}}\,, \label{metricAdS}
\end{equation}
where $g_{\mathbb{S}^{n-2}}$ is the canonical metric on the unit
$(n-2)$-dimensional sphere, associated with the coordinate~$\te$.

In this space we can picture the spatial limit $r\to+\infty$
as the cylinder $\RR_t\times \mathbb{S}^{n-2}$ with the standard
metric
\[
g_0:=dt^2-g_{\mathbb{S}^{n-2}}\,,
\]
which defines the timelike conformal infinity of the spacetime.

Let us now focus on the Klein--Gordon equation on this anti-de Sitter space
with the natural Dirichlet datum
\[
	\lim_{r\to\infty}r^{\frac n2-\al}\phi(t,r,\theta)=f(t,\theta),
\]
$f\in C^\infty_0(\RR^{n+1})$ and trivial initial conditions at time $-\infty$.
Upon expanding the operator $\Box_{g^+}$
associated to the metric \eqref{metricAdS},
we obtain the Klein--Gordon equation
\begin{align}
\begin{split}
\partial_{tt}\phi=\dfrac{1+r^2}{r^2}\De_\theta \phi&+(1+r^2)^2\partial_{rr} \phi\\
&+(1+r^2)\Big(\dfrac{n-1}{r}+(n+1) r \Big) \partial_r\phi -(1+r^2)\Big(\al^2-\dfrac{n^2}{4} \Big)\phi\,.
\end{split}
\end{align}

As before, in order obtain the scattering operator we take
a basis $\{\mathcal{Y}_j\}_{j\in\mathbb{N}}$
of spherical harmonics of energy $\lambda_j^2:=j(j+n-3)$. They satisfy
the equation
\[
-\Delta_{\theta} \mathcal{Y}_j=\lambda_j^2 \mathcal{Y}_j,
\]
with $\Delta_{\theta}$ the Laplace-Beltrami operator on $\mathbb{S}^{n-2}$.
Introducing the coefficients
\[
\phi_j(t,r):=\int_{\mathbb S^{n-2}}\phi(t,r,\te)\,\mathcal{Y}_j(\te)\, d\te\,,
\]
one can apply the Laplace transform and expand the Klein--Gordon
equation in the basis of spherical harmonics
to obtain an ODE in the variable $r$,
\begin{equation}
(1+r^2) \pd_{rr}\Phi_{j}+\Big(\frac{n-1}{r}+(n+1) r\Big)\pd_{r} \Phi_j+
\Big(\frac{s^2}{1+r^2}+\frac{\lambda_j ^2}{r^2}+\al ^2-\frac{n^2}{4}\Big)\Phi_j=0\,,
\end{equation}
where $$\Phi_j(s,r):=\int_{t_0}^\infty e^{-s(t-t_0)}\ \phi_j(t,r)\
dt$$
is the Laplace transform of $\phi_j(r,\te)$.

The explicit solution of this equation is a combination of certain powers of $r$
multiplied by ordinary hypergeometric functions. Discarding the
solution that is not locally in $H^1$ at the origin, $r=0$, we then obtain
\[
\Phi_j(s,r)=c_j(s)r^{\beta} e^{-\tfrac{i}{2} \log(1+r^2)^s}{}_2 F_1\Big( \tfrac{1}{2} (\be-is+\tfrac{n}{2}-\al),
\tfrac{1}{2} (\be-is+\tfrac{n}{2}+\al), \beta+\tfrac{n}{2},-r^2\Big),
\]
where  $$ \beta:=\frac{1}{2} \left(2-n+\sqrt{4 \lambda
    ^2+(n-2)^2}\right).$$
The coefficient $c_j(s)$ is readily computed using that
$F_j(s):=	\lim_{r\to\infty}r^{\frac n2-\al}\Phi_j(s,r)$
must be the Laplace transform of the $j\th$ component of the boundary
datum.

The transformed Dirichlet-to-Neumann operator is then readily shown to
be
\[
\Lambda_\al F_j(s):=\lim_{r\to\infty}r^{1+2\al}\pd_r (r^{\frac n2-\al}\Phi_j(s,r)).
\]
Using now the explicit solution, we obtain
that
\begin{align*}
\Lambda_\al F_j(s)=\dfrac{\Gamma (-\alpha )\Gamma\big(\frac{1}{2} (\be-is+\frac{n}{2}+\al)\big)
\Gamma \big(\frac{1}{2} (\be+is+\frac{n}{2}+\al)\big) }
{\Gamma (\alpha )\Gamma\big(\frac{1}{2} (\be-is+\frac{n}{2}-\al)\big)
\Gamma\big(\frac{1}{2} (\be+is+\frac{n}{2}-\al)\big)}
(\be-i s +\tfrac{n}{2}-\al)F_j(s) \label{DNads},
\end{align*}
which can be written using the Fourier transform in time
as
\[
	\widehat{\Lambda f}_j(\tau):=\lim\limits_{\ep\to 0^+}\Lambda_\al F_j(\ep+i\tau).
\]
It should be noticed that in the limit
of large frequencies of the multiplier $\sigma_\al(\tau,\lambda_j)$ is, up to some numerical factor, the principal symbol
of the scattering operator in this globally defined AdS space, as one
would have expected.

\section*{Acknowledgements}

A.E.\ and B.V.\ are supported by the ERC Starting Grant~633152 and by the
ICMAT--Severo Ochoa grant
SEV-2015-0554. M.d.M.G.\ is supported by Spanish national
project MTM2014-52402- C3-1-P and is part of the the Catalonian research group 2014SGR1083 and the Barcelona Graduate School of Math.

 \end{document}